\documentclass{amsart}

\usepackage{amsmath, amssymb, amsthm, hyperref}

\hypersetup{colorlinks=true, citecolor=magenta, linkcolor=blue}

\theoremstyle{plain}
\newtheorem{theorem}{Theorem}[section]
\newtheorem{lemma}[theorem]{Lemma}
\newtheorem{proposition}[theorem]{Proposition}

\theoremstyle{definition}
\newtheorem{definition}[theorem]{Definition}

\begin{document}

\title[Additive-Multiplicative Magic Squares]{Some Thoughts on the Search for $5 \times 5$ and $6 \times 6$ Additive-Multiplicative Magic Squares}

\author{Desmond Weisenberg}
\email{desmondweisenberg@gmail.com}


\begin{abstract}
An \emph{additive-multiplicative magic square} is a square grid of numbers whose rows, columns, and long diagonals all have the same sum (called the magic sum) and the same product (called the magic product). There are numerous open problems about magic squares by Christian Boyer on multimagie.com. One such problem is to construct or prove the impossibility of a $5 \times 5$ or $6 \times 6$ additive-multiplicative magic square of distinct positive integers. Here, we present a possible approach to this problem and some partial results. We observe that such a square can be described by a form determined by the prime factorizations of its entries and that identifying these forms might be helpful in finding such a square or ruling out specific magic products.
\end{abstract}

\maketitle

\section{Introduction}

An \emph{additive-multiplicative magic square} is a square grid of numbers whose rows, columns, and long diagonals all have the same sum (called the magic sum) and the same product (called the magic product). (Note that long diagonal refers to the two diagonals of the square going from the upper-left element to the lower-right element and from the upper-right element to the lower-left element.) There are numerous open problems about magic squares by Christian Boyer on multimagie.com, and currently prizes of \texteuro1,000 and \texteuro500 along with a bottle of champagne are available to whoever can construct or prove the impossibility of a $5 \times 5$ or $6 \times 6$ additive-multiplicative magic square, respectively, of distinct positive integers \cite{Multimagie}.

Similar results are known for other sizes of additive-multiplicative magic squares of distinct positive integers. The $1 \times 1$ case is trivial, and it is known that there are no such squares of sizes 2, 3, or 4. In contrast, there are known examples of sizes 7, 8, and 9 \cite{Multimagie}. In fact, it might even be reasonable to conjecture that such a square exists for all sizes greater than this --- however, the cases where the size is 5 or 6 still remain unsolved. To get a sense of the difficulty of these problems, see Figure \ref{ExampleNearMiss} for a $5 \times 5$ square that is ``nearly'' additive-multiplicative but narrowly fails. An example of a $7 \times 7$ additive-multiplicative magic square is also provided in Section \ref{Pairwise Zones}, Figure \ref{ExampleAddMult}.

\begin{figure}[htb]
\begin{center}
\begin{tabular}{|c|c|c|c|c|}
\hline
105 & 182 & 40 & 198 & 45 \\
\hline
78 & 216 & 66 & 175 & 35 \\
\hline
220 & 42 & 65 & 63 & 180 \\
\hline
140 & 55 & 189 & 30 & 156 \\
\hline
27 & 75 & 210 & 104 & 154 \\
\hline
\end{tabular}
\caption{Discovered by Lee Morgenstern, this $5 \times 5$ square of distinct positive integers is ``nearly'' additive-multiplicative. All of its rows, all of its columns, and the upper-left to lower-right long diagonal have a sum of 570 and a product of $6810804000 = 2^5 * 3^5 * 5^3 * 7^2 * 11 * 13$. However, the upper-right to lower-left long diagonal does not have the same sum or product \cite{Multimagie}. \label{ExampleNearMiss}}
\end{center}
\end{figure}

Here, we present a possible approach to this problem and some partial results. Our approach is to observe that such a square can be described by a form determined only by the prime factorizations of its entries and that identifying these forms might be helpful in finding or eliminating such a square. We outline this idea fully in Section \ref{Square Forms}. Next, in Section \ref{Unacceptability Results}, we prove several significant results based on this idea. Though we focus on the $5 \times 5$ and $6 \times 6$ cases, many of the ideas and results of this paper are applicable to larger sizes as well.

\section{Square Forms}
\label{Square Forms}

Recall that every positive integer $P$ can be uniquely (up to order) written as $p_1^{n_1}p_2^{n_2}\cdots p_k^{n_k}$ for some nonnegative integer $k$, distinct primes $p_1, p_2, \cdots, p_k$, and positive integers $n_1, n_2, \cdots, n_k$. The factors of $P$ are the numbers of the form $p_1^{q_1}p_2^{q_2}\cdots p_k^{q_k}$ where $0 \leq q_i \leq n_i$ for all $i$. If we wish to consider the factors of $P$, then a helpful approach may be to not consider the numerical value of $P$ itself, but to instead consider the \emph{prime signature} of $P$ --- that is, the multiset $\{n_1, n_2, \cdots, n_k\}$. We can use the prime signature to write a form for the number with its prime factors replaced by variables so we can consider its factors. For example, $64,800 = 2^53^45^2$ has the form $a^5b^4c^2$, so each of its factors thus has the form $a^{q_1}b^{q_2}c^{q_3}$ where $0 \leq q_1 \leq 5$, $0 \leq q_2 \leq 4$, and $0 \leq q_3 \leq 2$.

How does this help us? Our goal is to find a $5 \times 5$ or $6 \times 6$ additive-multiplicative magic square of distinct positive integers. If we temporarily drop the additive requirement, we can construct \emph{forms} of magic squares based on the prime signature of the magic product as described above. To do so, we use the following definition. 

\begin{definition}
A \emph{square form} is a square grid of products of powers of variables (representing distinct prime factors) such that all elements are distinct and each row, column, and long diagonal has the same product. Throughout this paper, we will take square forms to have size $5 \times 5$ or $6 \times 6$.
\end{definition}

For example, a $6 \times 6$ square form with a magic product of $a^8b^5c^3d^2ef$ is given in Figure \ref{ExampleSquareForm}.

\begin{figure}[htb]
\begin{center}
\begin{tabular}{|c|c|c|c|c|c|}
\hline
$a^7$ & $adef$ & $b^4$ & $bd$ & $c^3$ & $1$ \\
\hline
$b^2e$ & $c$ & $a^5b$ & $a^2bf$ & $d^2$ & $abc^2$ \\
\hline
$bc$ & $a^2$ & $c^2d^2$ & $a^4be$ & $abf$ & $ab^2$ \\
\hline
$b^2d$ & $a^2c^2$ & $a^2cf$ & $b^2$ & $a^3be$ & $ad$ \\
\hline
$adf$ & $a^2b^4$ & $e$ & $ac^2$ & $b$ & $a^4cd$ \\
\hline
$c^2$ & $abd$ & $a$ & $acd$ & $a^4b^2$ & $ab^2ef$ \\
\hline
\end{tabular}
\caption{A $6 \times 6$ square form of distinct expressions with magic product $a^8b^5c^3d^2ef$. If these variables are taken to represent distinct primes, then the corresponding $6 \times 6$ magic square is multiplicative (but not necessarily additive) and contains distinct positive integers. \label{ExampleSquareForm}}
\end{center}
\end{figure}

One approach to finding a $5 \times 5$ or $6 \times 6$ additive-multiplicative magic square of distinct positive integers may be to construct such square forms and then search for distinct primes that can be assigned to their variables to make them additive. This approach could be helpful because when checking whether a given number can be the magic product of such a square, rather than computing arrangements of the factors of that specific number, we can instead find the square forms whose magic products have the appropriate prime signature and then see if assigning the prime factors of the given number to the variables of the square forms makes any of the squares additive. Then if we find nothing and we wish to check whether another number of the same prime signature can be a magic product, we just need to reuse the square forms we have already found.

After constructing such square forms, the question naturally turns to whether for a given square form there exist distinct primes that can be assigned to its variables to make it additive.

\begin{definition}
Define a square form as \emph{acceptable} if there exist distinct primes that can be assigned to its variables to make its rows, columns, and long diagonals have the same sum. Define a square form as \emph{unacceptable} otherwise.
\end{definition}

Obviously, it is unknown whether $5 \times 5$ or $6 \times 6$ acceptable square forms even exist, since the problem of finding an additive-multiplicative magic square of distinct entries in the first place is essentially equivalent to finding an acceptable square form. However, we can show that certain square forms are \emph{un}acceptable, allowing us to exclude them from our consideration.

\begin{proposition}
\label{ExampleSquareProposition}
The square form in Figure \ref{ExampleSquareForm} is unacceptable.
\end{proposition}

\begin{proof}
For the sake of contradiction, suppose there exist distinct primes $a,\allowbreak b,\allowbreak c,\allowbreak d,\allowbreak e,\allowbreak f$ that, assigned to this square form, make it additive. Consider the fourth row and the fourth column, and note that we need not include the term where the row and column intersect, as it appears on both sides of the equality and thus can be cancelled. This results in the following equality:
\begin{equation*}
b^2d + a^2c^2 + a^2cf + a^3be + ad = bd + a^2bf + a^4be + ac^2 + acd.
\end{equation*}
This can be rewritten as
\begin{equation*}
b^2d + a(ac^2 + acf + a^2be + d) = bd + a(abf + a^3be + c^2 + cd).
\end{equation*}
It follows that $b^2d \equiv bd \pmod{a}$. Since $a$ is prime and $b$ and $d$ are not multiples of $a$, we can divide by $b$ and $d$ on both sides of this equation to get $b \equiv 1 \pmod{a}$. Similarly to above, we can consider the terms in the fourth and sixth columns with no $a$ component to see that $b^2 + bd \equiv 1 \pmod{a}$. Since $b \equiv 1 \pmod{a}$, we can replace $b$ with $1$ in this equation to get $1^2 + 1d \equiv 1 \pmod{a}$, which can be simplified to $d \equiv 0 \pmod{a}$. This contradicts the assumption that $a$ and $d$ are distinct primes. Since we have reached a contradiction, the given square form is unacceptable.
\end{proof}

This is just one example; other unacceptability proofs for different square forms may include different techniques. Of course, formulating an individualized proof for each square form quickly becomes impractical. Though it may be worthwhile to think about more efficient methods to do this, another thing we can do is formulate more general results about the unacceptability of certain square forms.

\section{Unacceptability Results}
\label{Unacceptability Results}

In this section, we prove some notable results about square forms and unacceptability. In Section \ref{Pairwise Zones}, we introduce the concept of ``pairwise zones'' of a square. In Section \ref{Further Results}, we use pairwise zones, along with other tools, to prove our main results about the square forms themselves.

\subsection{Pairwise Zones}
\label{Pairwise Zones}

Before going into further results on unacceptability, it is helpful to introduce the concept of what we call \emph{pairwise zones}. When discussing additivity, we have the condition that the square's \emph{strips} --- that is, its rows, columns, and long diagonals --- have the same sum. (This discussion similarly applies to multiplicativity and the magic product.) Furthermore, if we take two collections of strips in a square such that each collection has the same number of strips, it is clear that both collections will have the same sum. Moreover, if we have an equality of two collections of strips and the same element appears in both collections, we can remove it from both sides and still maintain this equality. For example, consider the collection of the top two rows and the collection of the left two columns of a $6 \times 6$ square. Taking the elements of these collections and subtracting the elements that appear in both yields two rectangles, so these two rectangles are guaranteed to have the same sum. As stated, this idea requires that both collections have the same number of strips. We can also assume without loss of generality that no single strip appears in both collections, since if any strips do appear in both collections, we can just subtract their elements from both sides of the corresponding equality until we get an equality that corresponds to collections with no single strips that appear on both sides.

Finally, we need a condition to ensure that after subtracting each instance of an element that appears on both sides of the equality, all the remaining elements appear exactly once on their side. This ensures that our final equality is actually an equality between two sums of distinct elements, which is necessary to ensure that it corresponds to two zones of the square. To achieve this, it may be tempting to require that no element appear in more than one strip within the same collection. However, this condition is actually too strong. For example, consider two collections of two strips of a $6 \times 6$ square: let the first collection consist of the top row and the left column, and let the second collection consist of both long diagonals. Then by subtracting elements that appear on both sides of the corresponding equality, we gain two zones of the square that must have the same sum. The condition that no element appear in more than one strip within the same collection is not satisfied, though, since the element in the upper-left square appears in both strips in the first collection. As such, we instead impose the weaker condition that for every element $x$ in the square, the number of times $x$ appears in the first collection and the number of times $x$ appears in the second collection differ by at most one. This way, after subtracting every instance of $x$ that appears on both sides of the corresponding equality, $x$ will ultimately appear at most once in the equality.

This motivates the following definition:

\begin{definition}
Let $X$ and $Y$ be two disjoint nonempty zones of the square. Suppose there exist two collections of strips such that no single strip appears in both collections, both collections have the same number of strips, and for every element $x$ in the square, the number of times $x$ appears in the first collection and the number of times $x$ appears in the second collection differ by at most one. Suppose $X$ is the region of the square containing the elements that appear once more in the first collection than in the second collection, and suppose $Y$ is the region of the square containing the elements that appear once more in the second collection than in the first collection. Then $X$ and $Y$ are said to be \emph{pairwise zones}.
\end{definition}

Observe that if the square is additive, then by our previous reasoning, the elements in any two pairwise zones must have the same sum. Similarly, if the square is multiplicative, then the elements in any two pairwise zones must have the same product. This idea will be important for some of our results, since we will be able to prove that certain square forms are unacceptable by finding properties of pairwise zones $X$ and $Y$ that guarantee that $X$ and $Y$ cannot have the same sum no matter what prime factors are assigned to the square forms.

The additive-multiplicative magic square in Figure \ref{ExampleAddMult} can be used to demonstrate examples of pairwise zones. For example, consider the zones defined by the top two rows and the left two columns, excluding the elements that appear in both. These zones contain the elements given by $\{50, 90, 48, 1, 84, 16, 54, 189, 110,\allowbreak 6\}$ and $\{100, 2, 96, 60, 3, 63, 99, 180, 21, 24\}$. Our above reasoning indicates that these zones must have the same sum and product. Indeed, they both have a sum of 648 and a product of $1955476131840000 = 2^{15} * 3^{11} * 5^4 * 7^2 * 11$.

\begin{figure}[htb]
\begin{center}
\begin{tabular}{|c|c|c|c|c|c|c|}
\hline
126 & 66 & 50 & 90 & 48 & 1 & 84 \\
\hline
20 & 70 & 16 & 54 & 189 & 110 & 6 \\
\hline
100 & 2 & 22 & 98 & 36 & 72 & 135 \\
\hline
96 & 60 & 81 & 4 & 10 & 49 & 165 \\
\hline
3 & 63 & 30 & 176 & 120 & 45 & 28 \\
\hline
99 & 180 & 14 & 25 & 7 & 108 & 32 \\
\hline
21 & 24 & 252 & 18 & 55 & 80 & 15 \\
\hline
\end{tabular}
\caption{The first known $7 \times 7$ additive-multiplicative magic square of distinct positive integers. Discovered by S\'{e}bastien Miquel on August 15, 2016, this square has a magic sum of 465 and a magic product of $150885504000 = 2^{10} * 3^7 * 5^3 * 7^2 * 11$ \cite{Multimagie}. \label{ExampleAddMult}}
\end{center}
\end{figure}

\subsection{Further Results}
\label{Further Results}

We now move on to some unacceptability results for $5 \times 5$ and $6 \times 6$ square forms. We start with a straightforward lemma whose use will become clear in Lemma \ref{BiggerZoneLemma}.

This lemma uses the sum-of-divisors function $\sigma(n)$, which is defined for all positive integers $n$ as the sum of positive factors of $n$, including 1 and itself. An important result in elementary number theory states that if $p_1, \cdots, p_k$ are distinct primes and $n_1, \cdots, n_k$ are non-negative integers, then
\begin{equation*}
\sigma(p_1^{n_1}p_2^{n_2}\cdots p_k^{n_k}) = \left(\frac{p_1^{n_1+1}-1}{p_1-1}\right)\left(\frac{p_2^{n_2+1}-1}{p_2-1}\right)\cdots\left(\frac{p_k^{n_k+1}-1}{p_k-1}\right).
\end{equation*}
For example, $\sigma(12) = 1 + 2 + 3 + 4 + 6 + 12 = 28$, and since $12 = 2^2 * 3$, the above formula yields $(\frac{2^3 - 1}{2 - 1})(\frac{3^2 - 1}{3 - 1}) = (7)(4) = 28$.

This function will be used in the following lemma.

\begin{lemma}
\label{SigmaLemma}
Let $p_1, \cdots, p_k$ be distinct primes, and let $n_1, \cdots, n_k$ be positive integers. Then
\begin{equation*}
\sigma(p_1^{n_1-1}p_2^{n_2-1}\cdots p_k^{n_k-1}) < p_1^{n_1}p_2^{n_2}\cdots p_k^{n_k}.
\end{equation*}
\end{lemma}

\begin{proof}
If $p_1, \cdots, p_k$ are distinct primes and $n_1, \cdots, n_k$ are positive integers, then
\begin{equation*}
\sigma(p_1^{n_1-1}p_2^{n_2-1}\cdots p_k^{n_k-1}) = \left(\frac{p_1^{n_1}-1}{p_1-1}\right)\left(\frac{p_2^{n_2}-1}{p_2-1}\right)\cdots\left(\frac{p_k^{n_k}-1}{p_k-1}\right),
\end{equation*}
which is clearly less than $p_1^{n_1}p_2^{n_2}\cdots p_k^{n_k}$.
\end{proof}

One way of showing unacceptability of a square form is by selecting two pairwise zones $X$ and $Y$ and showing that they can never have the same sum regardless of what prime factors are assigned to the square form. One way of doing this is by proving that one zone (say, $Y$) is ``bigger'' than the other (in this case, $X$) in the sense that no matter what primes are chosen, the sum of elements in $Y$ is always going to be larger than the sum of elements in $X$. One approach to doing this is by constructing a function $f$ that maps the elements in $X$ to the elements in $Y$ and then proving that for any choice of primes, each element in the image of $f$ is greater than the sum of elements that map to it. This is exemplified in the following result.

\begin{lemma}
\label{BiggerZoneLemma}
Suppose a square form has two pairwise zones $X$ and $Y$ such that there exists a function $f: X \to Y$ where a) every $x \in X$ is mapped to a multiple of itself, and b) if any $y \in Y$ is mapped to by some $x \in X$ such that some prime factor is raised to the same power in both $x$ and $y$, then no other element maps to $y$. Then the square form is unacceptable.
\end{lemma}

\begin{proof}
We show that for any distinct primes assigned to such a form, the sum of all elements in $Y$ is greater than the sum of all elements in $X$. To do so, we show that every $y$ in the image of $f$ is greater than the sum of the elements that map to it. If $y$ is mapped to by just one element $x$, then $y$ must be greater than $x$ since $y$ is a multiple of $x$. Alternatively, if $y$ is mapped to by more than one element, then every element $x$ that maps to $y$ must be a factor of $y$ where every prime factor is raised to a lesser degree in $x$ than in $y$. It follows directly from Lemma \ref{SigmaLemma} that $y$ is greater than the sum of all such values of $x$. Therefore, $X$ and $Y$ cannot have the same sum, so the square form is unacceptable.
\end{proof}

As an example of zones satisfying the hypotheses of Lemma \ref{BiggerZoneLemma}, let $X = \{a^3b^5, \allowbreak a^3b^3c^2, a^4, ab^3c^4\}$, and let $Y = \{a^4b^5, b^2, ac, a^6b^4c^5\}$. (It is easy to confirm that these zones have the same product and no repeating elements.) Define $f: X \to Y$ as mapping $a^3b^5$ to $a^4b^5$ and everything else to $a^6b^4c^5$. Then every element in $X$ is mapped to a multiple of itself. Furthermore, $a^4b^5$ is mapped to by $a^3b^5$, and there is a prime factor (namely, $b$) that is raised to the same power in both $a^4b^5$ and $a^3b^5$, but the hypotheses are still satisfied since no other element maps to $a^4b^5$. In contrast, $a^6b^4c^5$ is mapped to by multiple elements, but none have $a$ raised to the power of 6, $b$ raised to the power of 4, or $c$ raised to the power of 5, so the hypotheses are still satisfied.

Another way of showing that two zones cannot have the same sum is by showing that they cannot be congruent modulo a certain number. One fundamental result of this kind occurs in pairwise zones $X$ and $Y$ where a prime factor $p$ attains its minimum order exactly once throughout these zones --- that is, where there exists some element $x$ in $X$ or $Y$ such that $p$ is raised to a lower power in $x$ than in any other element of $X$ or $Y$.

\begin{lemma}
\label{MinimumOrderLemma}
Suppose a square form has two pairwise zones $X$ and $Y$ such that some prime factor $p$ attains its minimum order exactly once in these zones; that is, suppose there exists some element $x$ in $X$ or $Y$ such that $p$ is raised to a lower power in $x$ than in any other element of $X$ or $Y$. Then the square form is unacceptable.
\end{lemma}

\begin{proof}
Without loss of generality, suppose $x \in X$. Define $k$ as the order of $p$ in $x$. Since every other element in $X$ has $p$ raised to the power of at least $k + 1$, it follows that the sum of elements in $X$ is congruent to $x \pmod{p^{k+1}}$. However, since \emph{every} element in $Y$ has $p$ raised to the power of at least $k + 1$, it follows that the sum of elements in $Y$ is congruent to $0 \pmod{p^{k+1}}$. Since $x \not\equiv 0 \pmod{p^{k+1}}$, $X$ and $Y$ cannot have the same sum. As such, the square form is unacceptable.
\end{proof}

At this point, we have two lemmas that rule out the acceptability of a square form when certain patterns appear in the pairwise zones. Using these lemmas, we can move on to proving some broader unacceptability theorems.

\begin{theorem}
\label{PrimePowerTheorem}
If a square form has magic product $P$ and an entry $E$ such that $P/E$ is a prime power, then the square form is unacceptable.
\end{theorem}

\begin{proof}
Consider the pairwise zones that are the row and column that contain $E$ minus the entry of $E$ itself. Since these pairwise zones must multiply to $P/E$ and $P/E$ is a prime power, the zones must consist of distinct powers of a single prime (possibly including to the power of zero). Clearly, the highest of these powers must be in one of these zones. Since every element in the other zone divides this maximum power and has a smaller exponent, it follows from Lemma~\ref{BiggerZoneLemma} that the square form is unacceptable.
\end{proof}

\begin{lemma}
\label{FourTimesLemma}
If a square form has a prime factor that attains its minimum order fewer than four times throughout the square, then the square form is unacceptable.
\end{lemma}

\begin{proof}
For the sake of contradiction, suppose $p$ is a prime factor in an acceptable square form that attains its minimum order $k$ fewer than four times throughout the square. Clearly, $p$ attains order $k$ at least once. Define $C_1$ as a column containing a term where $p$ is raised to the $k$th power. Lemma \ref{MinimumOrderLemma} guarantees that for all other columns $C_2$, $p$ must attain order $k$ at least twice throughout $C_1$ and $C_2$. Since there are at most three entries with $p$ raised to the $k$th power, there are not enough to have one in every column; as such, to satisfy Lemma~\ref{MinimumOrderLemma}, $p$ must attain order $k$ at least twice in $C_1$. That is, if $p$ attains order $k$ in a column, it must do so twice in that column. It similarly holds that if $p$ attains order $k$ in a row, it must attain order $k$ at least twice in that row. Since $p$ attains order $k$ at least once in the square, it must therefore attain this order at least twice in one column --- that is, in at least two rows --- and at least twice in each of these rows. Therefore, it actually does attain this order at least four times, contradicting our initial assumptions.
\end{proof}

\begin{theorem}
If a square form has a magic product of the form $a^n$ (for any positive integer $n$), $a^nb^m$ for $1 \leq m \leq 3$, or $a^nbc$, then the square form is unacceptable.
\end{theorem}

\begin{proof}
For the sake of contradiction, suppose there exists an acceptable square form with such a product. Define $k$ as the lowest order $a$ attains at any entry in the square. If the magic product is of the form $a^n$ or $a^nb^m$ for $1 \leq m \leq 3$, then in an entry where $a$ is raised to the $k$th power, $a^k$ can be multiplied by nothing, $b$, or $b^2$; Theorem \ref{PrimePowerTheorem} implies that it cannot be multiplied by $b^3$. As such, $a$ can attain order $k$ in at most three entries.

Similarly, if the magic product is of the form $a^nbc$, then in an entry where $a$ is raised to the $k$th power, $a^k$ can be multiplied by nothing, $b$, or $c$ (but not $bc$), so $a$ can once again attain order $k$ in at most three entries. Either way, $a$ attains its minimum order at most three times, which contradicts Lemma \ref{FourTimesLemma}.
\end{proof}

Another interesting class of results deals not with unacceptability itself per se, but restricts the possible numerical values that the prime factors in a given square form can assume to make the square additive. We prove one major lemma along these lines and then show three significant results that follow as special cases of this lemma. Here, it helps to talk about the \emph{components} of an element of a square; let the non-$p$ component of an element refer to the product of all the prime powers of the element aside from $p$ itself. For example, $a^3b^5c^2$ and $a^8b^5c^2$ are different elements, but their non-$a$ components still have the same value, as they are both $b^5c^2$. The following lemma uses this idea to put an upper bound on the values that a prime factor can assume to make the square additive when certain conditions are satisfied.

\begin{lemma}
\label{PrimeValueLemma}
For positive integers $m$ and $n$, suppose a region of a square form can be partitioned into $m + n$ disjoint zones such that the square can only be additive if each zone has the same sum, and suppose there exists a prime factor $p$ such that the non-$p$ components in this region assume at most $m$ different values. Then the square can only be made additive if $p$ is set to a value no greater than $\lceil m/n \rceil$.
\end{lemma}

\begin{proof}
For every value of non-$p$ component that appears in this region, the region must have a maximum element with that non-$p$ component --- that is, an element such that no other element in the region has the same non-$p$ component and has $p$ raised to a higher power. As such, since the non-$p$ components in this region assume at most $m$ different values, there exist at most $m$ elements that are maximal with respect to their non-$p$ component. The region is partitioned into $m + n$ disjoint zones, and these maximal elements clearly can appear in at most $m$ of the zones, so define $Y$ as the union of $m$ zones such that each maximal element is in $Y$, and define $X$ as the union of the remaining $n$ zones.

Observe that the sum of the elements in $X$ must be $n$ times the sum of each individual zone, and the sum of the elements in $Y$ must be $m$ times the sum of each individual zone. As such, if we define $X'$ as the set of elements in $X$ but where each is multiplied by $\lceil m/n \rceil$, then the sum of elements in $X'$ must be at least the sum of elements in $Y$. Define the function $f: X' \to Y$ such that each element in $X'$ is mapped to the maximal element in $Y$ with the same non-$p$ component. Since the sum of elements in $X'$ is at least the sum of elements in $Y$, there is at least one element in the image of $f$ that is less than or equal to the sum of elements that map to it. Denote this element as $p^ky$, where $k$ is the order of $p$ in the element and $y$ is the non-$p$ component. Since the elements of $X'$ that map to $p^ky$ are all of the form $\lceil m/n \rceil p^{k_i}y$ where the $k_i$ are distinct and less than $k$ itself, and since $p^ky$ is less than or equal to the sum of elements that map to it, we have that
\begin{equation*}
p^ky \leq \lceil m/n \rceil p^{k_1}y + \cdots + \lceil m/n \rceil p^{k_j}y,
\end{equation*}
where $j$ is the number of elements that map to $p^ky$. Dividing $y$ from both sides, this gives us
\begin{equation*}
p^k \leq \lceil m/n \rceil p^{k_1} + \cdots + \lceil m/n \rceil p^{k_j}.
\end{equation*}

Now, we just need to prove that this inequality cannot be true if $p$ is greater than $\lceil m/n \rceil$. Since $k$ is greater than each $k_i$ and the $k_i$ are distinct, this follows simply from viewing the above expressions as base-$p$ representations of numbers.
\end{proof}

The remaining theorems on $5 \times 5$ and $6 \times 6$ square forms are all special cases of Lemma \ref{PrimeValueLemma}.

\begin{theorem}
Suppose a square form has magic product $P$ and a diagonal element $D$ such that $P/D$ is of the form $a^nb$ for some distinct primes $a$ and $b$. Then the square can only be additive if $a = 2$.
\end{theorem}

\begin{proof}
If such an element $D$ exists in a square form, there must exist three disjoint zones that multiply to $a^nb$ and have the same sum: the row containing $D$ excluding $D$ itself, the column containing $D$ excluding $D$ itself, and the long diagonal containing $D$ excluding $D$ itself. The elements in these zones must all be of the form $a^k$ or $a^kb$, so there are at most two possible non-$a$ components. Since the hypotheses of Lemma \ref{PrimeValueLemma} are satisfied with 3 zones and 2 possible non-$a$ components, we have that $a \leq 2$. Clearly, this means $a$ must be equal to 2.
\end{proof}

Since odd-sized squares have a center cell, we can use a similar idea to get another result in the $5 \times 5$ case.

\begin{theorem}
If a $5 \times 5$ square form has magic product $P$ and a center element $C$ such that $P/C$ is of the form $a^nb$ for some distinct primes $a$ and $b$, then the square form is unacceptable. Furthermore, if a $5 \times 5$ square form has magic product $P$ and a center element $C$ such that $P/C$ is of the form $a^nb^2$ for some distinct primes $a$ and $b$, then the square can only be additive if $a$ is equal to 2 or 3.
\end{theorem}

\begin{proof}
First, suppose $P/C$ is of the form $a^nb$. Then there must exist four disjoint zones that multiply to $a^nb$ and have the same sum: the row containing $C$ excluding $C$ itself, the column containing $C$ excluding $C$ itself, and the two long diagonals containing $C$ excluding $C$ itself. The elements in these zones must all be of the form $a^k$ or $a^kb$, so there are at most two possible non-$a$ components. Since the hypotheses of Lemma \ref{PrimeValueLemma} are satisfied with 4 zones and 2 possible non-$a$ components, we have that $a \leq 1$. However, it is impossible for a prime number to be less than or equal to 1, so the square form is unacceptable.

Next, suppose $P/C$ is of the form $a^nb^2$. Then there must exist four disjoint zones that multiply to $a^nb^2$ and have the same sum: the row containing $C$ excluding $C$ itself, the column containing $C$ excluding $C$ itself, and the two long diagonals containing $C$ excluding $C$ itself. The elements in these zones must all be of the form $a^k$, $a^kb$, or $a^kb^2$, so there are at most three possible non-$a$ components. Since the hypotheses of Lemma \ref{PrimeValueLemma} are satisfied with 4 zones and 3 possible non-$a$ components, we have that $a \leq 3$. Clearly, this means $a$ must be equal to 2 or 3.
\end{proof}

\begin{theorem}
If a $6 \times 6$ square form has a magic product $P$ of the form $a^nb^4$, then the square can only be additive if $a = 2$. Furthermore, if a $6 \times 6$ square form has a magic product $P$ of the form $a^nb^5$ or $a^nb^2c$, then the square can only be additive if $a = 2$, $3$, or $5$. Finally, if a $5 \times 5$ square form has a magic product $P$ of the form $a^nb^4$, then the square can only be additive if $a$ is equal to $2$ or $3$.
\end{theorem}

\begin{proof}
First, consider a $6 \times 6$ square form where $P = a^nb^4$. Then for the square form to be acceptable, the elements in the square must all be of the form $a^k$, $a^kb$, $a^kb^2$, or $a^kb^3$. (Theorem \ref{PrimePowerTheorem} implies that they cannot be of the form $a^kb^4$.) Consider the zones given by the six rows of the square. Since the hypotheses of Lemma \ref{PrimeValueLemma} are satisfied with 6 zones and 4 possible non-$a$ components, we have that $a \leq 2$. Clearly, this means $a$ must be equal to 2.

Next, consider a $6 \times 6$ square form where $P$ is $a^nb^5$ or $a^nb^2c$. If $P = a^nb^5$, then for the square form to be acceptable, the elements in the square must all be of the form $a^k$, $a^kb$, $a^kb^2$, $a^kb^3$, or $a^kb^4$ (but not $a^kb^5$). If $P = a^nb^2c$, then for the square form to be acceptable, the elements in the square must all be of the form $a^k$, $a^kb$, $a^kb^2$, $a^kc$, or $a^kbc$ (but not $a^kb^2c$). Again, consider the zones given by the six rows of the square. Since in both cases, the hypotheses of Lemma \ref{PrimeValueLemma} are satisfied with 6 zones and 5 possible non-$a$ components, we have that $a \leq 5$. Clearly, this means $a$ must be equal to 2, 3, or 5.

Finally, consider a $5 \times 5$ square form where $P = a^nb^4$. Then for the square form to be acceptable, the elements in the square must all be of the form $a^k$, $a^kb$, $a^kb^2$, or $a^kb^3$ (but not $a^kb^4$). Consider the zones given by the five rows of the square. Since the hypotheses of Lemma \ref{PrimeValueLemma} are satisfied with 5 zones and 4 possible non-$a$ components, we have that $a \leq 4$. As such, $a$ must be equal to 2 or 3.
\end{proof}

\section{Future Work}

At this point, we have shown that an acceptable square form must have all of its pairwise zones satisfy certain properties. We have also shown that an acceptable square form cannot have a magic product $P$ and an element $E$ such that $P/E$ is a prime power, cannot have a magic product $P$ and a center element $C$ such that $P/C$ is a prime power multiplied by another prime, nor can it have a magic product of the form $a^n$, $a^nb^m$ for $1 \leq m \leq 3$, or $a^nbc$. Finally, we have found cases in which we can restrict the numerical values that the prime factors in a square form can assume to make the square additive.

An approach to making further progress could be to try to expand these results --- for example, one ambitious goal might be to prove that an acceptable square form would have to have at least three distinct prime factors, assuming this is even true. Possible other tools for proving unacceptability results might include polynomial factorization or the AM-GM inequality. Yet another approach may be to develop algorithms or further techniques for proving the unacceptability of individual square forms, as was done in Proposition \ref{ExampleSquareProposition}.

\section{Acknowledgments}

Thanks to Dr. Nick Gurski of Case Western Reserve University's Department of Mathematics, Applied Mathematics, and Statistics for reviewing an earlier draft of this paper and encouraging me to submit it for publication during my time as an undergraduate.


\begin{thebibliography}{99}

\bibitem{Multimagie} Boyer, Christian. \emph{Multimagie.com}. multimagie.com. Accessed 24 March 2023.

\end{thebibliography}
\end{document}